\documentclass[12pt,reqno]{amsart}

\usepackage{a4wide}
\usepackage{needspace}
\usepackage{hyperref}
\usepackage[inline]{enumitem}

\usepackage{tikz}
\usetikzlibrary{decorations.pathreplacing, tikzmark}
\usetikzlibrary{matrix,fit,positioning,calc}
\usetikzlibrary{positioning}
\usepackage{amsmath}

\newtheorem{theorem}[subsection]{Theorem}
\newtheorem{lemma}[subsection]{Lemma}

\newtheorem{proposition}[subsection]{Proposition}

\newtheorem*{claim*}{Claim}

\providecommand{\Z}{\mathbb{Z}}
\providecommand{\N}{\mathbb{N}}
\providecommand{\C}{\mathbb{C}}
\providecommand{\R}{\mathbb{R}}
\providecommand{\Q}{\mathbb{Q}}
\providecommand{\F}{\mathbb{F}}
\providecommand{\E}{\mathop{\mathbb{E}}}

\providecommand{\wh}{\widehat}
\providecommand{\wt}{\widetilde}
\newcommand{\dd}{\,\mathrm{d}}
\providecommand{\Span}{\mathop{\rm Span}\nolimits}

\providecommand{\supp}{\mathop{\rm supp}\nolimits}
\renewcommand{\Re}{\mathop{\rm Re}\nolimits}

\newcounter{constcntbig} 
\newcounter{constcntlittle} 

\makeatletter
\newcommand{\newconstbig}[1]{%
    \refstepcounter{constcntbig}%
    \hypertarget{const:#1}{\mbox{}}%
    \protected@write\@auxout{}%
        {\string\newlabel{const:#1}{{\arabic{constcntbig}}{\thepage}}}%
    C_{\arabic{constcntbig}}%
}
\makeatother
\makeatletter
\newcommand{\newconstlittle}[1]{%
    \refstepcounter{constcntlittle}%
    \hypertarget{const:#1}{\mbox{}}%
    \protected@write\@auxout{}%
        {\string\newlabel{const:#1}{{\arabic{constcntlittle}}{\thepage}}}%
    c_{\arabic{constcntlittle}}%
}
\makeatother

\newcommand{\refconstlittle}[1]{%
    \hyperlink{const:#1}{c_{\ref{const:#1}}}%
}
\newcommand{\refconstbig}[1]{%
    \hyperlink{const:#1}{C_{\ref{const:#1}}}%
}

\numberwithin{equation}{section}

\begin{document}

\title{Modular Schur numbers}

\author{Tom Sanders}
\address{Mathematical Institute\\
University of Oxford\\
Radcliffe Observatory Quarter\\
Woodstock Road\\
Oxford OX2 6GG\\
United Kingdom}
\email{tom.sanders@maths.ox.ac.uk}
\begin{abstract}
We study modular analogues of Schur numbers for systems of linear equations. We show that these only depend on the number of equations, not their coefficients and in the case of one equation show stronger bounds.
\end{abstract}

\maketitle

\section{Introduction}

The Schur number, $f(r)$, is the largest natural number $N$ such that there is an $r$-colouring of $[N]:=\{1,\dots,N\}$ without a monochromatic solution to $x+y=z$. This was shown to exist by Schur in \cite{sch::4}. There is a modular analogue $h(r)$ which is the same but instead of asking for no monochromatic solutions to $x+y=z$ it asks for none to $x+y \equiv z \pmod {N+1}$. We look modulo $N+1$ here so that $[N]$ is a set of non-zero residues.

It is immediate that $h(r)\leq f(r)$ for all $r \in \N^*$ and as a result of work of Abbott and Wang \cite{Abbott:1977aa} and Heule \cite{heu::} we know there is equality for $r\in \{1,2,3,4,5\}$; in \cite[Problem I, p12]{Abbott:1977aa} Abbott and Wang conjectured that there is equality for all $r$.

Schur's result is generalised by Rado's theorem. To state Rado's theorem recall that an $n \times m$ matrix $A$ satisfies the columns condition over a field $\F$ if there is a partition $\mathcal{P}=\{P_1,\dots,P_d\}$ of the columns of $A$ such that
\begin{equation*}
\sum_{c \in P_j}{c} \in \Span_\F(\{c: c \in P_1\cup \dots\cup P_{j-1}\})\text{ for all }1\leq j \leq d
\end{equation*}
with the convention that this span is the zero space for $j=1$.
\begin{theorem}[Rado's theorem, {\cite[Satz IV, p445]{rad::1}}]
Suppose that $A$ is an $n \times m$ matrix of integers that satisfies the columns condition over $\Q$ and $[N]$ is $r$-coloured such that there is no monochromatic solution to $Ax=0$. Then $N=O_{A,r}(1)$.
\end{theorem}
Write $\mathfrak{R}_A(r)$ for the largest $N$ such that there is an $r$-colouring of $[N]$ with no monochromatic solution to $Ax=0$. In this notation Schur's function $f(r)=\mathfrak{R}_{A}(r)$ where $A=(\begin{array}{ccc}1 & 1 & -1\end{array})$.

The matrix $A=(\begin{array}{ccc}a & 1 & -1\end{array})$ satisfies the columns condition, but to have $ax_1+x_2-x_3=0$ with $x_1,x_2,x_3 \in [N]$ we must have $N \geq a+1$, and so certainly $f_a(r):=\mathfrak{R}_{A}(r)\geq a$. In particular, this means that there is no upper bound on $\mathfrak{R}_A(r)$ that depends only on ($r$ and) the number of rows of $A$.  On the other hand there is a modular setting where there is such a bound:
\begin{theorem}\label{thm.pofc}
Suppose that $p$ is prime, $A$ is an $n \times m$ matrix that satisfies the columns condition over $\Z/p\Z$, and $(\Z/p\Z)^*$ is $r$-coloured such that there is no monochromatic solution to $Ax=0$. Then $p=O_{n,r}(1)$.
\end{theorem}
We record a proof of this in \S\ref{sec.deuber}. 

The columns condition over $\Z/p\Z$ is different to that over $\Q$: if, for example, $A=(\begin{array}{ccc}1 & p & -p\end{array})$ for a prime $p$ then $A$ satisfies the columns condition over $\Q$, but not over $\Z/p\Z$. In this case the $1$-colouring of $(\Z/p\Z)^*$ will have no monochromatic solutions to $Ax=0$ for the trivial reason that $0$ is not coloured, but $p$ could be arbitrarily large, so we genuinely need the columns condition over $\Z/p\Z$ not over $\Q$ in Theorem \ref{thm.pofc}.

There is no upper bound on $p$ that is independent of $n$. This can be seen by considering a $(k+1)$-point Brauer configuration, which is a $(k+1)$-tuple $x=(x_1,x_1+x_{k+1},\dots,x_1+(k-1)x_{k+1},x_{k+1})$ or, equivalently $x$ with $Ax=0$ where
\[
A =
\begin{pmatrix}
-1 & 1 & 0 & 0 & \cdots & 0 & -1 \\
-1 & 0 & 1 & 0 & \cdots & 0 & -2 \\
-1 & 0 & 0 & 1 & \cdots & 0 & -3 \\
\vdots & & & & \ddots & & \vdots \\
-1 & 0 & 0 & 0 & \cdots & 1 & -(k-1)
\end{pmatrix}.
\]
Here $A$ satisfies the columns condition (over $\Q$ and $\Z/p\Z$) and if $k=p$ then $A$ has $p-1$ rows, and for $p$ prime any solution to $Ax=0$ has $x_i=0$ for some $i$ by the pigeonhole principle. Since no colouring of $(\Z/p\Z)^*$ colours $0$, $x$ is not monochromatic. It follows that $p \geq n+1$ in Theorem \ref{thm.pofc}.

The proof of Theorem \ref{thm.pofc} uses the Hales-Jewett theorem and so the bounds are poor, and it seems quite possible that $p=\exp(r^{O_n(1)})$ is true.  For $n=1$ we are able to show this:
\begin{theorem}\label{thm.main2}
Suppose that $p$ is prime, $A$ is a $1 \times m$ matrix that satisfies the columns condition over $\Z/p\Z$, and $(\Z/p\Z)^*$ is $r$-coloured such that there is no monochromatic solution to $Ax=0$. Then $p=\exp((2r)^{O(1)})$.
\end{theorem}
Our proof uses Fourier analysis very much in the style of the proofs of \cite[Theorem 24]{shk::7}, or the arguments of \cite{cwasch::0}, or \cite[\S2]{chapre::}.  Part of this can be described as using an inverse theorem for the Gowers $U^2$-norm relative to linear level sets (aka Bohr sets), and this suggests the generalisation to higher order Fourier analysis. In this direction Prendiville in \cite{Prendiville:2024aa} established an inverse theorem for the Gowers $U^3$-norm relative to quadratic level sets which he used to prove an analogous bound for $4$-point Brauer configurations in a toy setting. This is some evidence towards the $n=2$ case of Theorem \ref{thm.pofc} being true with the bounds of Theorem \ref{thm.main2}.

To bring us back to the start of the introduction, the same result that will give us Theorem \ref{thm.main2} will also give the following:
\begin{theorem}\label{thm.main3}
Suppose that $a, N \in \N$ are coprime, and $(\Z/N\Z)^*$, meaning the non-zero elements of $\Z/N\Z$, is $r$-coloured such that there is no monochromatic solution to $ax_1+x_2-x_3=0$. Then $N\leq \exp((2r)^{\newconstbig{2}})$.
\end{theorem}
The function $f_a(r)$ defined just before Theorem \ref{thm.pofc} has $f_1(r)=f(r)$ and we can also define the modular version $h_a(r)$ with $h_1(r)=h(r)$, by letting $h_a(r)$ be the largest $N\in \N^*$ such that there is an $r$-colouring of $[N]$ with no monochromatic solutions to $ax+y\equiv z\pmod{N+1}$. 

A natural extension of Abbott and Wang's conjecture would be $f_a(r)=h_a(r)$ for all $r$, but in fact $f_a(2) \neq h_a(2)$ for all $a>1$. This can be shown by a direct combinatorial argument\footnote{Writing $N:=a^2+3a$, if $(x,y,z)$ equals $(1,1,a+1)$, $(a+1,a+1,(a+1)^2)$, $(a+2,1,(a+1)^2)$, $(1,N-a+2,1)$ or $(a+1,a+2,N-a+2)$, then $ax+y\equiv z \pmod {N+1}$. Given a $2$-colouring of $\{1,\dots,N\}$ with no monochromatic solutions to $ax+y\equiv z \pmod {N+1}$, the first two triples ensure that $1$ and $(a+1)^2$ are red, say, and $a+1$ is blue. The third triple then ensures that $a+2$ is blue. The fourth triple then means that $N-a+2$ is blue, and the fifth that it is red, a contradiction.}, but for $a$ large enough (meaning $a^2+3a+1> \exp(4^{\refconstbig{2}})$) we can also use Theorem \ref{thm.main3}: We know from \cite[Theorem 9.11, p239]{lanrob::} that $f_a(2)=a^2+3a$, and so if $h_a(2)=f_a(2)$ we have $(h_a(2)+1,a)=1$ and we can apply Theorem \ref{thm.main3} with $r=2$ and $N=h_a(2)+1$ to get a contradiction. 

Notwithstanding the fact it is false, there are good reasons not to conjecture the natural extension above. Abbott and Wang offered evidence \cite[Problem I, p12]{Abbott:1977aa} for their conjecture including the observation that the colourings which are used to establish the values for $f(1)$, $f(2)$, $f(3)$, and $f(4)$ are also free of monochromatic modular solutions. This is not true for the corresponding colourings in \cite[Theorem 9.11, p239]{lanrob::} establishing the values of $f_a(2)$ for $a>1$.

\subsection*{Notation} We use big-$O$ and big-$\Omega$ notation throughout with the latter in the sense of Knuth \cite{knu::}. We shall also use $C_1,C_2,\dots>1$ and $c_1,c_2,\dots \in (0,1]$ as absolute constants which will be the same throughout the paper. This will help us make it clear that some choices are not circular.

\section{Proof of Theorem \ref{thm.pofc}}\label{sec.deuber}

In this section we shall prove Theorem \ref{thm.pofc}, which is a special case of the next theorem.
\begin{theorem}\label{thm.ff}
Suppose that $\F$ is a finite field, $\F^*$ is $r$-coloured, $A$ is an $n \times m$ matrix satisfying the columns condition over $\F$, and there is no $x \in \F^m$ with $Ax=0$ and $x_1,\dots,x_m$ all the same colour. Then $|\F| = O_{n,r}(1)$.
\end{theorem}
We shall use a slight variant of Deuber's $(m,p,c)$-sets \cite[(5), p111]{deu::}: For $F \subset \F$, $m \in \N$, and $t \in \F^{m}$ we write
\begin{equation*}
S(m,F; t):=\bigcup_{j=1}^m{S(m,F;t,j)} \text{ where }S(m,F;t,j):=t_j + F.t_{j+1}+ \cdots + F.t_m.
\end{equation*}
This is essentially an $(m-1,p,1)$-set generated by $t$ with the intervals $\{-p,\dots, p\}$ replaced by a more general set $F$. We shall not be concerned with the `$c$' in $(m,p,c)$-sets because we are working over a field which will make all non-zero $c$s equivalent.

First, as in Deuber's approach to Rado's theorem, we connect the columns condition over $\F$ to these sets $S(m,F;t)$ in the following result, proved in \S\ref{sec.cc}:
\begin{proposition}[Proposition \ref{prop.deub}]
Suppose that $A$ is a $k\times l$ matrix that satisfies the columns condition over $\F$. Then there is $d \leq k$ and $F \subset 
\F$ of size at most $(k+1)d^2$, such that for every $t \in \F^{d}$ there is $x \in S(d,F;t)^l$ such that $Ax=0$.
\end{proposition}
Secondly, we need a version of Deuber's theorem for sets of the form $S(m,F;t)$.  To state this we need some terminology which will let us pull back a colouring of $\F$ to a high-dimensional product. We will then be able to apply the Hales-Jewett theorem to this high-dimensional product to get the result.

Given $F \subset \F$ and $I$ a finite set we say that $t\in \F^I$ is \textbf{$F$-independent} if
\begin{equation*}
\sum_{i \in I}{f_i.t_i} = 0_\F \text{ for some }f=(f_i)_{i \in I} \in F^I \text{ implies }f_i =0 \text{ for all }i \in I.
\end{equation*}
For $J \subset I$ we also write $t_J:=\sum_{j \in J}{t_j}$, and so in particular if $1 \in F$ and $t$ is $F$-independent then $t_J=0$ implies $J=\emptyset$.

Finally, for $F \subset \F$ we write $F^0:=\{1\}$ and $F^{i+1}=F.F^i$ for $i \in \N_0$. We can now state the result we use, which we prove in \S\ref{sec.pe}.
\begin{proposition}[Proposition \ref{prop.end}]
For $m,r,s \in\N^*$ there is $\Psi(m,r,s) \in \N^*$ such that if $F \subset \F$ with $0,1 \in F$ has size at most $s$, $\F^*$ is $r$-coloured, $t \in \F^I$ is $F^{mr}$-independent, and $|I| \geq \Psi(m,r,s)$, then there is $t' \in \F^m$ such that $S(m,F;t')$ is monochromatic. 
\end{proposition}
With these we can prove the main result of the section.
\begin{proof}[Proof of Theorem \ref{thm.ff}]
By Proposition \ref{prop.deub} there is $d \leq n$ and a set $F'$ of size at most $(n+1)d^2$, such that for every $t' \in \F^d$ there is $x \in S(d,F';t')^m$ with $Ax=0$. Let $F:=F' \cup \{0,1\}$. Let $I$ be a set of size $M:=\Psi(d,r,(n+1)d^2+2)$.

Pick $t_1,\dots,t_M$ independently and uniformly at random from $\F$. Let $\mathcal{F}^*$ be the set of $f \in (F^{dr})^I$ with not all $f_i$s equal to $0_\F$, and note that
\begin{equation*}
\E{\#\left\{ f\in \mathcal{F}^*:\sum_{i \in I}{f_i.t_i}=0_{\F}\right \}}= \frac{|\mathcal{F}|}{|\F|}.
\end{equation*}
If $|\F| > |\mathcal{F}^*|$ then there is a choice of $t \in \F^I$ that is $F^{dr}$-independent and we can apply Proposition \ref{prop.end} to get $t' \in \F^d$ such that $S(d,F;t')$ is monochromatic. However, $S(d,F';t') \subset S(d,F;t')$ and hence there is $x \in \F^m$ monochromatic such that $Ax=0$. This contradiction means that
\begin{equation*}
|\F|\leq |\mathcal{F}^*| = |(F^{dr})^I|-1 \leq (((n+1)d^2+2)^{dr})^{\Psi(d,r,(n+1)d^2+2) }-1=O_{n,r}(1),
\end{equation*}
giving the result.
\end{proof}

\subsection{Using the columns condition}\label{sec.cc}

\begin{proposition}\label{prop.deub}
Suppose that $A$ is a $k\times l$ matrix that satisfies the columns condition over $\F$. Then there is $d \leq k$ and $F \subset 
\F$ of size at most $(k+1)d^2$, such that for every $t \in \F^{d}$ there is $x \in S(d,F;t)^l$ such that $Ax=0$.
\end{proposition}
\begin{proof}
Since $A$ satisfies the columns condition over $\F$ there is a partition $\mathcal{P}=\{P_1,\dots,P_{d}\}$ of the columns of $A$ such that if we write
\begin{equation*}
V_j:=\Span_\F(\{c: c \in P_1\cup \dots \cup P_{j}\}) \text{ for all }0 \leq j \leq d.
\end{equation*}
then
\begin{equation*}
\sum_{c \in P_{j}}{c} \in V_{j-1} \text{ for all }1 \leq j \leq d.
\end{equation*}
Let $J:=\{1 \leq j \leq d: \dim V_j = \dim V_{j-1}\}$ and let $i_1<i_2<\dots<i_r$ be the elements of $[d]\setminus J$ in order. Since
\begin{equation*}
0<\dim V_{i_1}<\dim V_{i_2}<\dots <\dim V_{i_r}
\end{equation*}
and these are subspaces of $k$-dimensional space we have $r \leq k$. Write $i_0:=0$ and 
\begin{equation*}
P_j':=P_{i_j}\cup P_{i_j-1}\cup \dots \cup P_{i_{j-1}+1} \text{ for all }1 \leq j \leq r.
\end{equation*}
Then $A$ satisfies the columns condition over $\F$ w.r.t.\ the partition $\{P_1',\dots,P_{r}'\}$. It follows that we may assume that $d\leq k$.

Let $n_j:=\dim \Span_\F(P_j)$ and let $c_1,\dots,c_{n_j} \in P_j$ be a basis; write $R_j:=P_j\setminus \{c_1,\dots,c_{n_j}\}$. Since the vectors in $P_j$ are columns in a matrix with $k$ rows we have that $n_j \leq k$. Now replace $A$ by the matrix $A'$ with the columns $c_1,\dots,c_{n_j}$ in each $P_j$ left the same, and each family $R_j$ replaced by the single column $\sum_{c \in R_j}{c}$ (which is the zero column if there are no remaining columns, so in this case we add an extra column). In particular the resulting $A'$ has $q:=(n_1+1)+\cdots + (n_d+1)$ columns.

Any $x\in \F^q$ with $A'x=0$ has some $\wt{x} \in \F^l$ with $A\wt{x}=0$ and the entries of $\wt{x}$ are just the entries of $x$ repeated some number of times. In particular, if we can show the conclusion for $A'$ then we have it for $A$.

This new matrix $A'$ satisfies the columns condition over $\F$ and in matrix form this tells us that $A'W=0$ for
\[
W=
\begin{pmatrix}
\mathbf{1}_{n_1+1} & \alpha^{(2)}_1 & \alpha^{(3)}_1 & \cdots & \alpha^{(d)}_1 \\
0                & \mathbf{1}_{n_2+1} & \alpha^{(3)}_2 & \cdots & \alpha^{(d)}_2 \\
0                & 0                & \mathbf{1}_{n_3+1} & \cdots & \alpha^{(d)}_3 \\
\vdots           & \vdots           & \vdots           & \ddots & \vdots \\
0                & 0                & 0                & \cdots & \mathbf{1}_{n_d+1}
\end{pmatrix},
\]
Here $\mathbf{1}_{n_i+1}$ is the $(n_i+1) \times 1$ matrix (column vector) all of whose entries are $1$, and $\alpha_i^{(j)}$ is an $(n_i+1) \times 1$ matrix (column vector). Now let
\begin{equation*}
F:=\{\text{ entries in the column vector }\alpha_i^{(j)}: 1 \leq i < j \leq d\},
\end{equation*}
and write $w^{(1)},\dots,w^{(d)}$ for the columns of the matrix $W$.

Suppose that $t=(t_1,\dots,t_d) \in \F^{d}$. Then
\begin{equation*}
x:=t_{1}w^{(1)}+t_{2}w^{(2)} + \cdots +t_{d}w^{(d)} \in \Span(w^{(1)},\dots,w^{(d)})
\end{equation*}
and so $A'x=0$.  For $1 \leq i \leq q$ there is a unique $1 \leq j \leq d$ such that
\begin{equation*}
(n_1+1)+\cdots + (n_{j-1}+1) < i \leq (n_1+1)+\cdots + (n_{j}+1),
\end{equation*}
and by construction of $x$ we then have
\begin{equation*}
x_i \in S(d,F;t,j).
\end{equation*}
It follows that $x_i \in S(d,F;t)$ for all $1 \leq i \leq q$ as claimed. Finally it remains to note that $|F| \leq \sum_{i=1}^{d-1}{(d-i)(n_i+1)}\leq (k+1)d^2$.
\end{proof}

\subsection{Proving the analogue of Deuber's theorem}\label{sec.pe}
This is a simple adaptation of the proof of Deuber's theorem in \cite[\S4, p8]{gun::}, which also has an exposition at \cite{Moreira:2014aa}. That proof of Deuber's theorem uses the Hales-Jewett theorem, and to record that we need some notation. 

For finite sets $A$ and $B$ it is useful to write $A^B$ for the set of $B$-tuples of elements of $A$, so that $A^B$ is also finite and $|A^B|=|A|^{|B|}$. We use both the notation $(a_b)_{b \in B}$ and $a:B \rightarrow A$ for elements of $A^B$. 

\begin{theorem}[Hales-Jewett theorem]\label{thm.dhj}
For every $r,k \in \N^*$ there is a positive integer $H(r,k)$ such that for every $n \geq H(r,k)$, and all sets $A$ of size at most $k$ and $B$ of size at least $n$, for any $r$-colouring of $A^B$ there is $\emptyset \neq W \subset B$ and $z \in A^B$ such that
\begin{equation*}
\{x \in A^B: \exists a \in A \text{ with } x|_W \tikzmarknode{b1}{\equiv} a \text{ and }x|_{B \setminus W} \tikzmarknode{b2}{=} z|_{B \setminus W}\}
\end{equation*}
\begin{tikzpicture}[overlay, remember picture, >=stealth]

\node[font=\tiny, left=8cm of pic cs:b1, yshift=-0.5\baselineskip, align=left, red] (explain)
{ $x$ is takes the\\ constant value\\ $a$ on $W$};

\node[font=\tiny, right=5.5cm of pic cs:b2, yshift=-0.5\baselineskip, align=left, red] (explain2)
{ $x$ and $z$ agree\\  on $B \setminus W$.\\ value(s) of $z$\\ on $B$ irrelevant};

\draw[->, red,rounded corners=5pt]
  ([xshift=0.5cm]explain.east) -| (b1.south);

\draw[->, red, rounded corners=5pt]
  ([xshift=-0.5cm]explain2.west) -| (b2.south);
\end{tikzpicture}
is monochromatic.
\end{theorem}

\begin{lemma}\label{lem.proit}
For every $s,d,r \in \N^*$ there is $\Phi(s,d,r) \in \N^*$ such that if $1 \in F \subset \F$ has size at most $s$, $\F^*$ is $r$-coloured, $t \in \F^I$ is $F$-independent,  and $\mathcal{P}$ is a partition of a subset of $I$ with $|\mathcal{P}| \geq \Phi(s,d,r)$, then there is $\mathcal{E}\subset \mathcal{P}$ and a partition $\mathcal{P}'\cup \{\bigcup{\mathcal{E}}\}$ of a subset of $I$ such that
\begin{enumerate}[label=(\roman*)]
\item $|\mathcal{P}'| \geq d$;
\item every set in $\mathcal{P}'$ is a union of sets in $\mathcal{P}$;\footnote{$\mathcal{P}$ need not be a refinement of $\mathcal{P}'\cup \{\bigcup{\mathcal{E}}\}$ in the normal sense because the base set of $\mathcal{P}'\cup \{\bigcup{\mathcal{E}}\}$ may be smaller than that of $\mathcal{P}$.}
\item  there is $t_0' \in \sum_{P \in \mathcal{E}}{F.t_P}$ and a colour class $C$ with
\begin{equation*}
t_0'+\sum_{P \in \mathcal{P}'}{F.t_P} \subset C.
\end{equation*}
\end{enumerate}
\end{lemma}
\begin{proof}
Set $\Phi(s,d,r):=dH(r,s^d)+1$. Let $A:=F^{d}$ so $|A| \leq s^d$; and $\mathcal{Q}$ be a partition (of some subset of the set partitioned by $\mathcal{P}$) such that each $S \in \mathcal{Q}$ has $S=S_1\sqcup \cdots \sqcup S_{d}$ where $S_1,\dots,S_{d}\in\mathcal{P}$, and $|\mathcal{Q}| = H(r,s^{d})$; and let $Z \in \mathcal{P}$ be disjoint from $\bigcup{\mathcal{Q}}$. This is possible since $|\mathcal{P}| \geq \Phi(s,d,r)$.

\begin{figure}[h]
\centering
\begin{tikzpicture}[x=1cm,y=1cm]

\def\xA{0.0}
\def\xB{1.}
\def\xC{2}
\def\xD{3}
\def\xE{4}

\def\yA{0.0}
\def\yB{1.0}
\def\yC{2.}
\def\yD{2.}
\def\yE{3.}
\def\yF{4}

\draw[black, line width=0.7pt] (\xA,\yA) rectangle (\xE,\yF);
\draw[black, line width=0.7pt] (\xB,\yA) -- (\xB,\yF);
\draw[black, line width=0.7pt] (\xC,\yA) -- (\xC,\yF);
\draw[black, line width=0.7pt] (\xD,\yA) -- (\xD,\yF);

\draw[black, line width=0.7pt] (\xA,\yB) -- (\xE,\yB);
\draw[black, line width=0.7pt] (\xA,\yD) -- (\xE,\yD);
\draw[black, line width=0.7pt] (\xA,\yE) -- (\xE,\yE);

\draw[red, very thick, rounded corners=2pt, xshift=-2pt, yshift=-2pt] (\xA,\yA) rectangle (\xB,\yF);
\draw[red, very thick, rounded corners=2pt, xshift=-2pt, yshift=-2pt] (\xB,\yA) rectangle (\xC,\yF);
\draw[red, very thick, rounded corners=2pt, xshift=-2pt, yshift=-2pt] (\xC,\yA) rectangle (\xD,\yF);
\draw[red, very thick, rounded corners=2pt, xshift=-2pt, yshift=-2pt] (\xD,\yA) rectangle (\xE,\yF);

\node at ($( \xA,\yF)!0.5!(\xB,\yE)$) {$S_1$};
\node at ($( \xB,\yF)!0.5!(\xC,\yE)$) {$T_1$};
\node at ($( \xC,\yF)!0.5!(\xD,\yE)$) {$\cdots$};
\node at ($( \xD,\yF)!0.5!(\xE,\yE)$) {$\cdots$};

\node at ($( \xA,\yE)!0.5!(\xB,\yD)$) {$S_2$};
\node at ($( \xB,\yE)!0.5!(\xC,\yD)$) {$T_2$};
\node at ($( \xC,\yE)!0.5!(\xD,\yD)$) {$\cdots$};
\node at ($( \xD,\yE)!0.5!(\xE,\yD)$) {$\cdots$};

\node at ($( \xA,\yB)!0.5!(\xB,\yD)$) {$\vdots$};
\node at ($( \xB,\yB)!0.5!(\xC,\yD)$) {$\vdots$};

\node at ($( \xA,\yB)!0.5!(\xB,\yA)$) {$S_d$};

\draw[<->, >=stealth, line width=0.8pt]
  (-0.6,\yA) -- (-0.6,\yF)
  node[midway, left=6pt] {$d$ rows};

\draw[<->, >=stealth, line width=0.8pt]
  (\xA,\yF+0.6) -- (\xE,\yF+0.6)
  node[midway, above=6pt] {$H(r,s^d)$ columns};

\def\xR{5.2} 

\def\xS{9.2} 

\draw[black, line width=0.7pt]  (\xR,2.6) rectangle ++(1,1);
\draw[black, line width=0.7pt]  (\xR+1.8,2.6) rectangle ++(1,1);
\draw[black, line width=0.7pt]  (\xR+0.9,1.2) rectangle ++(1,1);

\node at (\xR+1.4,1.7) {$Z$};
\end{tikzpicture}
\caption{Black squares represent sets in $\mathcal{P}$; red rectangles represent sets in $\mathcal{Q}$.}
\end{figure}

If $x \in A^\mathcal{Q}$ then for $S \in \mathcal{Q}$, we write $x(S)=(x(S)_1,\dots,x(S)_d) \in F^{d}$. The map
\begin{equation*}
\phi:A^\mathcal{Q} \rightarrow \F^*; x \mapsto t_Z+\sum_{S \in \mathcal{Q}}{\sum_{i=1}^{d}{x(S)_it_{S_i}}}
\end{equation*}
genuinely maps \emph{into} $\F^*$ because $t$ is $F$-independent and $1 \in F$. This means we can use it to colour $A^\mathcal{Q}$ by pulling back the colouring of $\F^*$. This gives an $r$-colouring of $A^\mathcal{Q}$, and in view of the sizes of $\mathcal{Q}$ and $A$, by the Hales-Jewett Theorem (Theorem \ref{thm.dhj}) there is $z \in A^\mathcal{Q}$ and $\emptyset \neq \mathcal{W} \subset \mathcal{Q}$ such that 
\begin{equation*}
\mathcal{L}:=\{x \in A^\mathcal{Q}: \exists a \in A\text{ with } x|_\mathcal{W}\equiv a \text{ and }x|_{\mathcal{Q} \setminus \mathcal{W}} =z|_{\mathcal{Q} \setminus \mathcal{W}}\}
\end{equation*}
is monochromatic. Let
\begin{equation*}
\mathcal{E}:= \{Z\} \cup \{S_1,\dots,S_d: S\in \mathcal{Q}\setminus \mathcal{W}\} ,
\end{equation*}
and
\begin{equation*}
\mathcal{P}':=\left\{P_1,\dots,P_d\right\} \text{ where } P_i=\bigcup_{S \in \mathcal{W}}{S_i}.
\end{equation*}
Since $\mathcal{W}$ is non-empty the set $P_i$ is non-empty, and by design the $P_i$s are pairwise disjoint so $\mathcal{P}'$ is a partition of size $d$ as required. Since $1 \in F$,
\begin{equation*}
t_0':=t_Z+\sum_{S \in \mathcal{Q} \setminus \mathcal{W}}{\sum_{i=1}^d{z(S)_it_{S_i}}} \in \sum_{P \in \mathcal{E}}{F.t_P}.
\end{equation*}
Now, suppose that $a \in F^d$, so that there is $x \in \mathcal{L}$ with $x(S)_i=a_i$ for all $1 \leq i \leq d$ and $S \in \mathcal{W}$, and $x(S)_i=z(S)_i$ for all $1 \leq i \leq d$ and $S \in \mathcal{Q}\setminus \mathcal{W}$. Then
\begin{equation*}
\phi(x)=t_Z+\sum_{S \in \mathcal{Q} \setminus \mathcal{W}}{\sum_{i=1}^d{x(S)_it_{S_i}}} +\sum_{S \in \mathcal{W}}{\sum_{i=1}^d{x(S)_it_{S_i}}}=t_0'+\sum_{S \in \mathcal{W}}{\sum_{i=1}^d{a_it_{S_i}}} = t_0'+\sum_{i=1}^d{a_it_{P_i}}.
\end{equation*}
Since $a$ was arbitrary it follows that $t_0' + F.t_{P_1}+\cdots +F.t_{P_d}$ is monochromatic as claimed, and the result is proved.
\end{proof}

\begin{proposition}\label{prop.end}
For $m,r,s \in\N^*$ there is $\Psi(m,r,s) \in \N^*$ such that if $F \subset \F$ with $0,1 \in F$ has size at most $s$, $\F^*$ is $r$-coloured, $t \in \F^I$ is $F^{mr}$-independent, and $|I| \geq \Psi(m,r,s)$, then there is $t' \in \F^m$ such that $S(m,F;t')$ is monochromatic. 
\end{proposition}
\begin{proof}
Set $n_0:=1$ and
\begin{equation*}
n_{i+1}:=\Phi(s^{i+1},n_i,r) \text{ for }0 \leq i <mr.
\end{equation*}
Let $\Psi(m,r,s):=n_{mr}$. Let $\mathcal{P}^{(mr)}$ be the singletons of $I$, $E^{(mr)}:=\emptyset$ and $z^{(mr)}:=0_\F$.  We construct partitions $\mathcal{P}^{(i)}$ of size at least $n_i$, sets $E^{(i)}$, and elements $z^{(i)}$ by downward recursion.

Suppose that we are at step $i+1$.  Apply Lemma \ref{lem.proit} with the Lemma's $F$ and $\mathcal{P}$ equal to $F^{i+1}$ and $\mathcal{P}^{(i+1)}$ respectively. Since $i+1\leq mr$ and $1 \in F$ we have $F^{i+1} \subset F^{mr}$ and so $t$ is $F^{i+1}$-independent. Moreover,  $\mathcal{P}^{(i+1)}$ has size at least $n_{i+1}$. Hence the Lemma gives a set $\mathcal{E}^{(i)} \subset \mathcal{P}^{(i+1)}$ and a partition $\mathcal{P}^{(i)}\cup\{\bigcup{\mathcal{E}^{(i)} }\}$ with $|\mathcal{P}^{(i)}| \geq n_i$, in which each set in $\mathcal{P}^{(i)}$ is a union of sets in $\mathcal{P}^{(i+1)}$, and a
\begin{equation}\label{eqn.iy}
z^{(i)} \in \sum_{P \in \mathcal{E}^{(i)}}{F^{i}.t_P}
\end{equation}
with
\begin{equation}\label{ekey1}
z^{(i)}+\sum_{P \in \mathcal{P}^{(i)}}{F^{i}.t_P} \subset C^{(i)}
\end{equation}
for some colour class $C^{(i)}$.

By design,
\begin{equation}\label{ekey}
\text{ whenever }i<i', \text{the sets in } \mathcal{P}^{(i)} \text{ are unions of sets in } \mathcal{P}^{(i')}.
\end{equation}
We have $\mathcal{E}^{(i)}\subset \mathcal{P}^{(i+1)}$, and so by (\ref{ekey}), $\bigcup{\mathcal{E}^{(i)}}$ is a union of sets in $\mathcal{P}^{(i')}$. However, $\bigcup{\mathcal{E}^{(i')}}$ is disjoint from all sets in $\mathcal{P}^{(i')}$, and so
\begin{equation}\label{ekey5}
\bigcup{\mathcal{E}}^{(i)} \cap \bigcup{\mathcal{E}}^{(i')} = \emptyset\text{  whenever } i<i'.
\end{equation}

By the pigeonhole principle there is $0 \leq i_1 <i_2 <\dots <i_m \leq mr$ such that
\begin{equation*}
C^{(i_1)}=\cdots = C^{(i_m)}=:C.
\end{equation*}
Now suppose $1 \leq k\leq m$, and $1 \leq j <k$.  By (\ref{eqn.iy}), and then (\ref{ekey}) and the fact that both $\mathcal{E}^{(i_j)}$ and $\mathcal{P}^{(i_k)}$ are partitions, we have
\begin{equation*}
F.z^{(i_j)} \subset \sum_{P \in \mathcal{E}^{(i_j)}}{F^{i_j+1}.t_P} \subset \sum_{ P \in \mathcal{P}^{(i_k)}:P \subset \bigcup{\mathcal{E}^{(i_j)}}}{F^{i_j+1}.t_P}.
\end{equation*}
Since $1 \in F$ we have $F^{{i_j}+1} \subset F^{i_k}$, and since $0 \in F$ too, adding the above using the disjointness in (\ref{ekey5}) we get
\begin{equation*}
F.z^{(i_{k-1})}+\cdots + F.z^{(i_1)}  \subset \sum_{P \in \mathcal{P}^{(i_k)}}{F^{i_k}.t_P}.
\end{equation*}
Finally, by (\ref{ekey1}) we have
\begin{equation*}
z^{(i_k)} + F.z^{(i_{k-1})}+\cdots + F.z^{(i_1)}\subset z^{(i_k)} + \sum_{P \in  \mathcal{P}^{(i_k)}}{F^{i_k}.t_P} \subset C.
\end{equation*}
Set $t'_1:=z^{(i_m)}$, $t'_2:=z^{(i_{m-1})}$,\dots, $t_m':=z^{(i_1)}$ and then we have shown what is required.
\end{proof}

\section{Proofs of Theorems \ref{thm.main2} \& \ref{thm.main3}}\label{sec.notnetc}

The main result of this section is Theorem \ref{thm.maincore} from which we will establish Theorems \ref{thm.main2} \& \ref{thm.main3} at the end. To prove Theorem \ref{thm.maincore} we use standard Fourier arguments, the basics of which we first have to record. 

Throughout this section we write $G$ for a finite Abelian group (which we think of as topologically compact and Hausdorff), and $\Gamma$ for a dual group of $G$ (which we think of as topologically discrete).

The duality is denoted by the pairing $\langle x,\gamma\rangle$ for elements $x \in G$ and $\gamma \in \Gamma$. The homomorphisms $G \rightarrow S^1:=\{z \in \C: |z|=1\}$ are exactly the maps $x \mapsto \langle x,\gamma\rangle$, and similarly for the homomorphisms $\Gamma \rightarrow S^1$.

We work with measures on $G$ and we take them all to have the $\sigma$-algebra $\mathcal{P}(G)$. For any non-empty $B \subset G$ we write $\mu_B$ for the uniform probability measure on $G$ supported on $B$. Given measures $\mu$ and $\nu$ on $G$ we define
\begin{equation*}
\wt{\mu}(E):=\overline{\mu(-E)}\text{ and }\mu \ast \nu(E):=\int{1_E(x+y)\dd \mu(x)\dd\nu(y)} \text{ for all }E \subset G.
\end{equation*}
We norm these measures by $\|\mu\|:=\int{\dd|\mu|}$ which makes them into a Banach algebra. The Fourier-Stieltjes transform of $\mu$ at $\gamma \in \Gamma$ is defined to be
\begin{equation*}
\wh{\mu}(\gamma)=(\mu)^\wedge(\gamma):=\int{\overline{\langle x,\gamma\rangle}\dd\mu(x)}.
\end{equation*}
For $f,g \in L_1(\mu_G)$ and a measure $\mu$ we define
\begin{equation*}
f\ast g(x):=\int{f(y)g(x-y)\dd\mu_G(y)} \text{ and } \mu \ast f(x)=f \ast \mu(x):=\int{f(x-y)\dd\mu(y)},
\end{equation*}
for all $x \in G$, and the Fourier transform of $f$ at $\gamma \in \Gamma$ is defined to be
\begin{equation*}
\wh{f}(\gamma)=(f)^\wedge(\gamma):=\int{f(x)\overline{\langle x,\gamma\rangle}\dd \mu_G(x)}=(f\dd \mu_G)^\wedge(\gamma).
\end{equation*}
With an eye to the Fourier inversion formula, for $g \in \ell_1(\Gamma)$ and $x \in G$ we write
\begin{equation*}
\wh{g}(x):=\sum_{\gamma \in \Gamma}{g(\gamma)\langle x,\gamma\rangle}.
\end{equation*}
Rudin \cite{rud::1} covers the Fourier transform in the style of this notation, though it has many more details about the analysis; Tao and Vu \cite[Chapter 4]{taovu::} covers the facts we use about the Fourier transform and is closer in spirit to our work here.

For $\Lambda \subset \Gamma$ and $\delta >0$ we define the Bohr set with frequency set $\Lambda$ and width $\delta$ to be
\begin{equation*}
B(\Lambda,\delta):=\{x \in G: |\langle x,\lambda\rangle -1|<\delta \text{ for all }\lambda \in \Lambda\}.
\end{equation*}
This is a slightly different definition to that in \cite[Definition 4.16, p187]{taovu::}, but the two definitions are equivalent as we shall see in the proof of the next lemma.
\begin{lemma}[Sizes of Bohr sets]\label{lem.bohrsize}
Suppose that $\Lambda \subset \Gamma$ has size $d$ and $\delta \in (0,1]$. Then
\begin{equation*}
\mu_G(B(\Lambda,2\delta)) \leq \exp(\newconstbig{grow}d)\mu_G(B(\Lambda,\delta))\text{ and }\mu_G(B(\Lambda,\delta))\geq (\delta/2)^{O(d)}.
\end{equation*}
\end{lemma}
\begin{proof}
We write $\|x\|_{\R/\Z}:=\min\{|\theta|: \theta \in x\}$ where we think of $x$ as a coset of $\Z$, and then for $\gamma \in \Gamma$ and $x \in G$ we have
\begin{equation*}
4\left\|\frac{1}{2\pi}\log \langle x,\gamma\rangle\right \|_{\R/\Z} \leq |\langle x,\gamma\rangle-1| \leq 2\pi\left\|\frac{1}{2\pi}\log \langle x,\gamma\rangle\right \|_{\R/\Z} .
\end{equation*}
With this we can apply \cite[Lemma 4.19, p188]{taovu::} to get
\begin{equation*}
\mu_G(B(\Lambda,\delta)) \geq (\delta/2\pi)^d \text{ and } \mu_G(B(\Lambda,2\delta)) \leq 16^d\mu_G(B(\Lambda,\delta)).
\end{equation*}
The lemma is proved.
\end{proof}
This growth lets us identify pairs of Bohr sets that are well-behaved:
\begin{lemma}[Regularity of Bohr sets]\label{lem.regular} Suppose $\delta,\eta \in (0,1]$ and $\Lambda \subset \Gamma$ has size $d$. Then there is $\delta_* \in [\delta/2,\delta]$ and $\delta' \in  (\Omega(\delta \eta/d),\delta]$ such that
\begin{equation*}
\mu_G(B(\Lambda,\delta_*)+B(\Lambda,\delta'))\leq (1+\eta)\mu_G(B(\Lambda,\delta_*)) \text{ and } B(\Lambda,\delta') \subset B(\Lambda,\delta_*).
\end{equation*}
\end{lemma}
\begin{proof}
Let $k=O(\eta^{-1})$ be a natural number such that $\exp(\refconstbig{grow}/k) \leq 1+\eta$ (where $\refconstbig{grow}$ is as in Lemma \ref{lem.bohrsize}), and let $\delta':=\delta/2kd$ and $\delta_i:=\delta/2+i\delta'$ for $0 \leq i \leq kd-1$. Then by Lemma \ref{lem.bohrsize} we have
\begin{equation*}
\prod_{i=0}^{kd-1}{\frac{\mu_G(B(\Lambda,\delta_i+\delta'))}{\mu_G(B(\Lambda,\delta_i))}} \leq \frac{\mu_G(B(\Lambda,\delta))}{\mu_G(B(\Lambda,\delta/2))} \leq \exp(\refconstbig{grow}d).
\end{equation*}
By averaging there is some $\delta_*=\delta_i$ such that 
\begin{equation*}
\frac{\mu_G(B(\Lambda,\delta_*+\delta'))}{\mu_G(B(\Lambda,\delta_*))} \leq 1+\eta.
\end{equation*}
The result follows since $B(\Lambda,\delta_*)+B(\Lambda,\delta') \subset B(\Lambda,\delta_*+\delta')$.
\end{proof}
The previous lemma gives a plentiful supply of sets that behave enough like groups that many arguments that work for groups can be `localised' to this approximate setting.

Parseval's theorem is frequently used to bound the size of sets of large Fourier coefficients -- see \cite[p204, (4.38)]{taovu::} for the sort of thing we have in mind.  Green and Tao also note this at \cite[p108]{gretao::1}, where they additionally localised this technique. We record the result of their ideas below:

\begin{lemma}[Local Parseval bound]\label{lem.localbessel}
Suppose $\epsilon,\sigma,\eta,\delta \in (0,1]$,
\[
\begin{array}{l l}
 \mathord{\bullet}\,\mu_G({B_0}+B_1) \leq 2 \mu_G({B_0})& \mathord{\bullet}\, \mu_G(B_2+B_1)\leq (1+\eta)\mu_G(B_1)\;\tikzmark{start}
\end{array}
\]
\begin{tikzpicture}[overlay, remember picture, >=stealth]

\node[font=\tiny, right=3cm of pic cs:start,  align=left, red] (explain)
{ these measure\\  the quality \\ of the local\\ approximation};

\draw[->, red,rounded corners=5pt]
  ([yshift=2pt, xshift=-0.5cm]explain.west) -- ([yshift=2pt]pic cs:start);
\end{tikzpicture}
and $\mu_{B_0}(S)\geq \sigma$. Then there is a Bohr set $B_3$ with frequency set of size at most $O(\epsilon^{-2}\sigma^{-1})$ and width $\delta$ such that
\begin{equation}\label{eqn.estimate}
|1-\langle x,\gamma\rangle|\leq \newconstbig{12} \eta\epsilon^{-3}\sigma^{-3/2}+\delta \textrm{ for all }x \in (B_2-B_2)\cap B_3 
\end{equation}
whenever $|(1_S\dd\mu_{B_0})^\wedge(\gamma)| \geq \epsilon \mu_{B_0}(S)$. 
\end{lemma}
\begin{proof}
Let $\mu:=\mu_{{B_0}+B_1} \ast \mu_{-B_1}$ so that by design and hypothesis if $S':=S\cap B_0$ then $\mu(S') \geq \frac{1}{2}\sigma$ and
\begin{equation}\label{deltalive}
\Delta:=\left\{\gamma:|(1_S\dd\mu_{B_0})^{\wedge}(\gamma)| \geq \epsilon \mu_{B_0}(S)\right\} = \left\{\gamma : |(1_{S'}\dd\mu)^\wedge(\gamma)| \geq \epsilon \mu(S')\right\}.
\end{equation}
Let $k:=\lfloor 4\epsilon^{-2}\sigma^{-1}\rfloor$. We say that a set $\Lambda \subset \Delta$ containing $1_\Gamma$, is $K$-orthogonal if
\begin{equation}\label{eqn.ju}
\|\wh{g}\|_{L_2(\mu)} \leq (1+K)^{1/2}\|g\|_{\ell_2(\Lambda)} \text{ for all } g \in \ell_2(\Lambda).
\end{equation}
Let $\Lambda_1:=\{1_\Gamma\}$ which is $0$-orthogonal (and $1_\Gamma \in \Delta$). Suppose that we have defined $\Lambda_1,\dots,\Lambda_j$ such that $\Lambda_i$ is $\frac{i-1}{k}$-orthogonal for all $i \leq j$. If there is some $\gamma \in \Delta\setminus \Lambda_j$ such that $\Lambda_j \cup \{\gamma\}$ is $\frac{j}{k}$-orthogonal then let $\Lambda_{j+1}:=\Lambda_j \cup \{\gamma\}$; otherwise, terminate the iteration. 

Suppose that $\Lambda_j$ is defined for all $j \leq k+1$. Then $\Lambda_{k+1}$ is $1$-orthogonal. By (\ref{eqn.ju}) the map $\ell_2(\Lambda_{k+1}) \rightarrow L_2(\mu);g \mapsto \wh{g}$ has norm at most $\sqrt{2}$, and hence by duality the adjoint $L_2(\mu) \mapsto \ell_2(\Lambda_{k+1}); f \mapsto (f \dd \mu)^\wedge|_{\Lambda_{k+1}}$ also has norm at most $\sqrt{2}$. It follows from (\ref{deltalive}) that
\begin{equation*}
(k+1) \left(\epsilon \mu(S')\right)^2 \leq \sum_{\lambda \in \Lambda_k}{|(1_{S'}\dd\mu)^\wedge(\lambda)|^2} \leq 2 \|1_{S'}\|_{L_2(\mu)}^2 \leq 2\mu(S').
\end{equation*}
This contradicts the value of $k$, and hence the iterative construction of $\Lambda_j$ terminates for some $j\leq k$. In this case set $\Lambda:=\Lambda_j$, and let $B_3$ be the Bohr set with frequency set $\Lambda$ and width $\delta$.

Suppose that $\gamma \in \Delta$. Our aim is to show that (\ref{eqn.estimate}) holds.  If $\gamma \in \Lambda$ then we are certainly done so we may assume that $\gamma \in \Delta \setminus \Lambda$.  Since $\Lambda \cup \{\gamma\}$ is not $\frac{j}{k}$-orthogonal (since the iteration terminated), there is $g \in \ell_2(\Lambda)$ and $\nu \in \C$ such that
\begin{equation*}
\int{|\wh{g} + \nu \gamma |^2\dd\mu} > \left(1+\frac{j}{k}\right)(\|g\|_{\ell_2(\Lambda)}^2 + |\nu|^2).
\end{equation*}
Multiplying this out and using the $\frac{j-1}{k}$-orthogonality of $\Lambda$ we get that
\begin{eqnarray*}
2\Re \langle \wh{g},\nu\gamma \rangle_{L_2(\mu)} & > & \left(1+\frac{j}{k}\right)(\|g\|_{\ell_2(\Lambda)}^2 + |\nu|^2)- \left(1+\frac{j-1}{k}\right)\|g\|_{\ell_2(\Lambda)}^2- |\nu|^2\\ & \geq & \frac{1}{k}(\|g\|_{\ell_2(\Lambda)}^2 + |\nu|^2) \geq  \frac{2}{k}\|g\|_{\ell_2(\Lambda)}|\nu|.
\end{eqnarray*}
Since the inequality is strict $\nu \neq 0$ and so dividing out we get that 
\begin{equation*}
 \frac{1}{k}\|g\|_{\ell_2(\Lambda)} < |\langle \wh{g},\gamma \rangle_{L_2(\mu)}| = \left|\sum_{\lambda \in \Lambda}{g(\lambda)\wh{\mu}(\gamma -\lambda)}\right| \leq \|g\|_{\ell_1(\Lambda)}\sup_{\lambda \in\Lambda}{|\wh{\mu}(\gamma-\lambda)|}.
\end{equation*}
Cauchy-Schwarz tells us that $\|g\|_{\ell_1(\Lambda)} \leq \sqrt{k}\|g\|_{\ell_2(\Lambda)}$, and hence there is some $\lambda \in \Lambda$ such that
\begin{equation*}
 |(\mu_{B_1})^\wedge(\gamma-\lambda)|\geq |\wh{\mu}(\gamma-\lambda)|  \geq k^{-3/2}.
\end{equation*}
Now, suppose $x \in B_2-B_2$ so that there are $s,t \in B_2$ with $x=s-t$. Then
\begin{align*}
|\langle s,\lambda-\gamma\rangle-\langle t,\lambda-\gamma\rangle||(\mu_{B_1})^\wedge(\gamma-\lambda)| &= |(\mu_{t+{B_1}})^\wedge(\gamma-\lambda) - (\mu_{s+{B_1}})^\wedge(\gamma-\lambda)|\\ &  \leq \frac{\mu_G((t+{B_1}) \triangle (s+{B_1}))}{\mu_G({B_1})} \leq 2\eta.
\end{align*}
Since $|\langle s,\lambda-\gamma\rangle-\langle t,\lambda-\gamma\rangle|=|\langle x,\lambda\rangle-\langle x,\gamma\rangle|$ we therefore conclude that
\begin{equation*}
|\langle x,\lambda\rangle-\langle x,\gamma\rangle| \leq 2\eta k^{3/2} \textrm{ for all } x \in B_2-B_2,
\end{equation*}
which gives us (\ref{eqn.estimate}) by the triangle inequality since $\lambda \in \Lambda$. The result is proved.
\end{proof}

We can now establish our main iteration lemma.

\begin{lemma}[Iteration lemma]\label{lem.itlem}
Suppose that $\alpha,\sigma \in (0,1]$,
\[
\begin{array}{l l}
\mathord{\bullet}\,\mu_G(B_0+B_1) \leq (1+\newconstlittle{e}\alpha)\mu_G(B_0) & \mathord{\bullet}\, \mu_G(B_1+B_2)\leq 2\mu_G(B_1)\\[5pt]
\mathord{\bullet}\,\mu_G(B_2+B_3)\leq (1+\newconstlittle{f}(\alpha \sigma)^{\newconstbig{g}})\mu_G(B_2) & 
\end{array}
\]
and $\mu_{B_0}(A)=\alpha$, $x_0 \in B_1$, and $\mu_{B_1-x_0}(S)=\sigma$. Then either
\begin{enumerate}[label=(\roman*)]
\item we have
\begin{equation*}
\langle 1_S,1_A \ast 1_{-A}\rangle_{L_2(\mu_G)} \geq \frac{1}{2}\sigma \alpha^2\mu_G(B_1)\mu_G(B_0);
\end{equation*}
\item or there is a Bohr set $B_4$ with a frequency set of size $O(\alpha^{-2}\sigma^{-1})$ and width $\Omega(1)$ such that for any probability measure $\mu$ supported on $B_1\cap (B_3-B_3) \cap B_4$ we have
\begin{equation*}
\|1_A\ast \mu\|_\infty \geq (1+\newconstlittle{inc})\alpha.
\end{equation*}
\end{enumerate}
\end{lemma}
\begin{proof}
Let $\refconstlittle{f}$ and $\refconstbig{g}$ be absolute constants chosen so that
\begin{equation}\label{ey}
\refconstlittle{f}(\alpha \sigma)^{\refconstbig{g}} \leq \frac{1}{4} \cdot (8\alpha^{-1})^{-3}\cdot \sigma^{3/2} \cdot \refconstbig{12}^{-1} \text{ for all }\alpha,\sigma \in (0,1].
\end{equation}
The conclusions are monotonic in $A$, so we can certainly assume that $A \subset B_0$. For $x \in B_1-B_1$ we can write $x=s-t$ for $s,t \in B_1$ and hence
\begin{align*}
\alpha^21_{B_0} \ast 1_{-B_0}(x),\, &\alpha 1_{B_0} \ast 1_{-A}(x) , \text{ and }\alpha1_A \ast 1_{-B_0}(x)\\ & = \alpha^2 \mu_G(B_0)+ O\left(\alpha\mu_G((s+B_0) \triangle (t+B_0))\right).
\end{align*}
Now $\mu_G((s+B_0)\triangle (t+B_0)) \leq 2(\mu_G(B_1+B_0)-\mu_G(B_0))$, and so it follows that there is an absolute constant $\refconstlittle{e}>0$ such that if $\mu_G(B_0+B_1) \leq (1+\refconstlittle{e}\alpha)\mu_G(B_0)$ then 
\begin{align}
\nonumber \alpha 1_A \ast 1_{-B_0}(x) +\alpha 1_{B_0} \ast &1_{-A}(x) -\alpha^21_{B_0} \ast 1_{-B_0}(x)\\ & \label{eqn.ibv} \geq \left(1-\frac{1}{64}\right)\alpha^2\mu_G(B_0) \text{ for all }x \in B_1-B_1.
\end{align}
From (\ref{eqn.ibv}) if we are not in the first conclusion of the lemma then
\begin{equation*}
\left|\langle 1_S,(1_A - \alpha1_{B_0}) \ast (1_{-A} -\alpha1_{-B_0})\rangle_{L_2(\mu_G)}\right| > \left(\frac{1}{2}-\frac{1}{64}\right)\sigma \alpha^2\mu_G(B_1-x_0)\mu_G(B_0).
\end{equation*}
By Plancherel's theorem we have
\begin{equation*}
\sum_{\gamma \in \Gamma}{|(1_S\dd \mu_{B_1-x_0})^\wedge(\gamma)| |(1_A - \alpha1_{B_0})^\wedge(\gamma)|^2} \geq \frac{1}{4}\sigma \alpha^2\mu_G(B_0),
\end{equation*}
and
\begin{equation*}
\sum_{\gamma \in \Gamma}{ |(1_A - \alpha1_{B_0})^\wedge(\gamma)|^2} =\|1_A - \alpha1_{B_0}\|_{L_2(\mu_{G})}^2 = (\alpha-\alpha^2) \mu_G(B_0).
\end{equation*}
Hence, writing $\Delta:=\{\gamma: |(1_S\dd\mu_{B_1-x_0})^\wedge(\gamma)| \geq \frac{1}{8}\alpha\sigma\}$, we have
\begin{equation}\label{eqn.massfound}
\sum_{\gamma \in \Delta}{ |(1_A - \alpha1_{B_0})^\wedge(\gamma)|^2} \geq \frac{1}{8} \alpha^2\mu_G(B_0).
\end{equation}
Apply Lemma \ref{lem.localbessel} with the Lemma's $\epsilon$ equal to $\frac{1}{8}\alpha$; the Lemma's $\sigma$ equal to $\sigma$; the Lemma's $\delta$ equal to $\frac{1}{4}$; the Lemma's $\eta$ equal to $\refconstlittle{f}(\alpha \sigma)^{\refconstbig{g}} $; the Lemma's $B_0$ equal to $B_1-x_0$; the Lemma's $B_1$ equal to $B_2$; and the Lemma's $B_2$ equal to $B_3$. This gives us a Bohr set $B_4$ with a frequency set of size $O(\alpha^{-2}\sigma^{-1})$ and width $\frac{1}{4}$ such that 
\begin{equation*}
|1-\langle x,\gamma\rangle| \leq \refconstbig{12}\refconstlittle{f}(\alpha \sigma)^{\refconstbig{g}}\left(\frac{1}{8}\alpha\right)^{-3}\sigma^{-3/2} + \frac{1}{4} 
\end{equation*}
for all $\gamma \in \Delta$ and $x \in (B_3-B_3)\cap B_4$. From  (\ref{ey}) it follows that
\begin{equation*}
|1-\langle x,\gamma\rangle| \leq \frac{1}{2} \text{ for all }\gamma \in \Delta \text{ and } x\in (B_3-B_3)\cap B_4.
\end{equation*}
Suppose that $\mu$ is a probability measure supported on $B_1\cap (B_3-B_3)\cap B_4$. Then by the triangle inequality $|\wh{\mu}(\gamma)| \geq \frac{1}{2}$ for all $\gamma \in \Delta$. Hence from (\ref{eqn.massfound}) we have
\begin{equation*}
\sum_{\gamma \in \Gamma}{|\wh{\mu}(\gamma)|^2 |(1_A - \alpha1_{B_0})^\wedge(\gamma)|^2} \geq\frac{1}{32} \alpha^2\mu_G(B_0).
\end{equation*}
By Plancherel's theorem it then follows that
\begin{equation*}
\langle (1_A - \alpha1_{B_0})\ast (1_{-A}  - \alpha1_{-B_0}), \mu \ast \wt{\mu}\rangle \geq \frac{1}{32} \alpha^2 \mu_G(B_0).
\end{equation*}
Since the support of $\mu$ is contained in $B_1$, the support of $\mu \ast \wt{\mu}$ is contained in $B_1-B_1$, and hence by (\ref{eqn.ibv}) we have
\begin{equation*}
\langle 1_A \ast 1_{-A}, \mu \ast \wt{\mu}\rangle \geq \left(1+\frac{1}{64}\right) \alpha^2 \mu_G(B_0).
\end{equation*}
Since $\int{1_A\ast \mu \dd\mu_G} = \alpha \mu_G(B_0)$ this gives the second conclusion by the triangle inequality.
\end{proof}

With this iteration lemma we can prove our main result of the section:

\begin{theorem}\label{thm.maincore}
Suppose that $\phi:G \rightarrow G$ is an automorphism and $G$ is $r$-coloured. Then there are at least $\exp(-r^{O(1)})|G|^2$ monochromatic triples $(x,y,z)$ with $\phi(x)=y-z$.
\end{theorem}
\begin{proof}
First we set some notation. Write $A_1,\dots,A_r$ for the colour classes, and for $\gamma \in \Gamma$ we write $\phi_*(\gamma)$ for the element of $\Gamma$ corresponding to the homomorphism $G \rightarrow S^1; x\mapsto \langle \phi(x),\gamma\rangle$.

We proceed iteratively to construct sets $\Lambda_1 \subset \dots \subset \Gamma$ of sizes $d_1 \leq d_2\leq \dots$ respectively, and reals $1 \geq \delta_1 \geq \delta_2 \geq \dots >0$. We put
\begin{equation*}
B^{(i)}:=\bigcap_{j=0}^{i-1}{\phi^{-j}(B(\Lambda_i,\delta_i))} = B(\wt{\Lambda_i},\delta_i) \text{ where }\wt{\Lambda_i}:=\Lambda_i\cup \phi_*(\Lambda_i) \cup \cdots \cup \phi_*^{i-1}(\Lambda_i).
\end{equation*}
In particular,
\begin{equation*}
\wt{\Lambda_1}\subset \wt{\Lambda_2}\subset \dots \subset \Gamma \text{ and }|\wt{\Lambda_i}| \leq id_i.
\end{equation*}
We initialise with $\Lambda_1:=\{1_\Gamma\}$, $d_1:=1$, and $\delta_1:=1$. For each step $i$ of the iteration, and $j \in [r]$ write
\begin{equation*}
S_{i,j}:=\min{\{\|1_{A_j} \ast \mu\|_{\infty} : \mu \text{ is a probability measure supported in } B^{(i)}\}}.
\end{equation*}
From the monotonicity of the $\delta_i$s and the nesting of the $\wt{\Lambda_i}$s, the Bohr sets $B^{(i)}$ are nested, and hence
\begin{equation}\label{mon2}
S_{i,j} \geq S_{k,j}\text{ whenever }i \geq k, \text{ and } S_{i,j} \leq 1 \text{ for all }i.
\end{equation}
At step $i$ we shall show that either
\begin{enumerate}
\item\label{cd1} there is some $j \in [r]$ such that
\begin{equation*}
\langle 1_{A_j} \ast 1_{-A_j},1_{\phi( A_j)}\rangle_{L_2(\mu_G)} \geq \exp(-r^{O(1)});\text{ or }
\end{equation*}
\item \label{cd0} there is $j \in [r]$ such that
\begin{equation*}
 d_{i+1} \leq d_i + r^{O(1)}\text{ and }\delta_{i+1} \geq  (1/2id_ir)^{O(1)} \delta_i
\end{equation*}
and for which
\begin{enumerate}
\item\label{cd22} $S_{i+1,j} \geq (1+\Omega(1))S_{i,j}$ and $S_{i,j} \geq 1/2r$;
\item\label{cd23} or $S_{i,j} < 1/2r$ and $S_{i+1,j}\geq 1/2r$.
\end{enumerate}
\end{enumerate}
We stop the iteration the first time we are in case (\ref{cd1}). By monotonicity (\ref{mon2}) we can be in case (\ref{cd23}) at most once for each $j \in [r]$. Suppose that we are in case (\ref{cd22}) for a particular $j$ at steps $i_1 <i_2<\dots < i_k$ of the iteration. Then by monotonicity (\ref{mon2}) we have
\begin{align*}
1 \geq S_{i_k+1,j} \geq (1+\Omega(1))S_{i_k,j} & \geq (1+\Omega(1))S_{i_{k-1}+1,j}\\ & \geq \qquad \dots \qquad  \geq (1+\Omega(1))^kS_{i_1,j} \geq (1+\Omega(1))^k\cdot (1/2r).
\end{align*}
It follows that $k=O(\log 2r)$. Since there are $r$ possible values for $j$ we conclude that we must have terminated the iteration at step $i_0\leq r\cdot (1+O(\log 2r))$, in which case
\begin{equation}\label{eqn.bounds}
|\wt{\Lambda_{i_0}}| \leq i_0d_{i_0} =i_0^2\cdot r^{O(1)} = r^{O(1)} \text{ and }\delta_{i_0} = \prod_{i<i_0}{\left(\frac{1}{rid_ir}\right)^{O(1)}}=\exp(-r^{O(1)})
\end{equation}
by (\ref{cd0}) and the fact that $\delta_1=\Omega(1)$.

Since $|G|^2\langle 1_{A_j} \ast 1_{-A_j},1_{\phi(A_j)}\rangle_{L_2(\mu_G)} $ is exactly the number of solutions to $\phi(x)=y-z$ with $x,y,z \in A_j$ we have the conclusion of the Theorem from (\ref{cd1}).

It remains to show that at each stage of the iteration we are either in case (\ref{cd1}), (\ref{cd22}) or (\ref{cd23}). Suppose we are at step $i$. By Lemma \ref{lem.regular} there is a Bohr set $B_0 \subset B^{(i)}$ of width $\Omega(\delta_i)$ and frequency set $\wt{\Lambda_i}$, and another $B_1'$ of width $\Omega(\delta_i/rid_i)$ and frequency set $\wt{\Lambda_i}$ such that
\begin{equation*}
\mu_G(B_0+B_1')\leq \left(1+\frac{\refconstlittle{e}}{2r}\right)\mu_G(B_0) \text{ and }B_1' \subset B_0.
\end{equation*}
By Lemma \ref{lem.regular} again there is a Bohr set $B_1\subset B_1'$ of width $\Omega(\delta_i/rid_i)$ and frequency set $\wt{\Lambda_i}$, and another $B_2'$ of width $\Omega(\delta_i/r^2(id_i)^2)$ and frequency set $\wt{\Lambda_i}$ such that
\begin{equation}\label{eqn.u}
\mu_G(B_1+B_2')\leq \left(1+\frac{1}{4r}\right)\mu_G(B_1) \text{ and }B_2' \subset B_1.
\end{equation}
Finally, by Lemma \ref{lem.regular} again there is a Bohr set $B_2 \subset B_2'$ of width $\Omega(\delta_i/r^2(id_i)^2)$ and frequency set $\wt{\Lambda_i}$, and another $B_3$ of width $\Omega(\delta_i/r^{O(1)}(id_i)^3)$ and frequency set $\wt{\Lambda_i}$  such that
\begin{equation*}
\mu_G(B_2+B_3)\leq (1+\refconstlittle{f}(1/4r^2)^{\refconstbig{g}})\mu_G(B_2) \text{ and }B_3 \subset B_2.
\end{equation*}
By averaging there is some $j$ such that 
\begin{equation}\label{eqn.lowi}
\mu_G(\phi(A_j)\cap B_1)\geq \frac{1}{r}\mu_G(B_1).
\end{equation}

Suppose that $S_{i,j}<1/2r$. Let $\delta_{i+1}$ be the width of $B_2$ and $\Lambda_{i+1}$ be $\Lambda_i$. Suppose that $\mu$ has support in $B^{(i+1)}$. Write $\nu$ for the pushforward measure defined by $\nu(E):=\mu(\phi^{-1}(E))$ for all $E\subset G$. Then
\begin{equation*}
\supp \nu \subset \phi(\supp \mu)\subset \bigcap_{j=-1}^{i-1}{\phi^{-j}(B(\Lambda_i,\delta_{i+1}))} \subset B(\wt{\Lambda_i},\delta_{i+1})=B_2.
\end{equation*}
Hence if $x$ is in the support of $\nu$, then by (\ref{eqn.u}) and using that the identity is in $B_2$, 
\begin{equation*}
\|\mu_{B_1+x} - \mu_{B_1}\|\leq  \frac{\mu_{G}( (B_1+x) \triangle B_1)}{\mu_G(B_1)} \leq \frac{1}{2r}.
\end{equation*}
By the triangle inequality we conclude that $\|\mu_{B_1} \ast \nu - \mu_{B_1}\|\leq 1/2r$ and hence by (\ref{eqn.lowi}) we have
\begin{align*}
\|1_{A_j} \ast \mu\|_\infty =\|1_{\phi(A_j)} \ast \nu\|_\infty&\geq \|1_{\phi(A_j)} \ast \mu_{B_1}\ast \nu\|_\infty\\& \geq \|1_{\phi(A_j)} \ast \mu_{B_1}\|_\infty - \frac{1}{2r} =\mu_{-B_1}(\phi(A_j)) -\frac{1}{2r} \geq \frac{1}{2r}.
\end{align*}
It follows that we are in case (\ref{cd23}).

Otherwise $S_{i,j} \geq 1/2r$. Let $x_1 \in G$ be such that $\alpha:=\mu_{x_1+B_0}(A)=\mu_{x_1-B_0}(A)=\|1_{A_j} \ast \mu_{B_0}\|_\infty \geq S_{i,j} \geq 1/2r$.  Apply Lemma \ref{lem.itlem} with the Lemma's $\alpha$ equal to $\alpha$; the Lemma's $\sigma$ equal to $\mu_{B_1}(\phi(A_j))\geq 1/r $; the Lemma's $B_0$ equal to $x_1+B_0$; the Lemma's $B_1$ equal to $B_1$; the Lemma's $B_2$ equal to $B_2$; the Lemma's $B_3$ equal to $B_3$; the Lemma's $A$ equal to $A_j$; the Lemma's $x_0$ equal to the identity; and the Lemma's $S$ equal to $\phi(A_j)$. Then either we have
\begin{equation*}
\langle 1_{\phi(A_j)}, 1_{A_j} \ast 1_{-A_j}\rangle_{L_2(\mu_G)} \geq \frac{1}{8r^3}\mu_G(B_1)\mu_G(B_0).
\end{equation*}
Lemma \ref{lem.bohrsize} and (\ref{eqn.bounds}) then tells us that we are in case (\ref{cd1}). If we are not in the first case of Lemma \ref{lem.itlem} then we get a Bohr set $B_4$ with frequency set $\Lambda'$ of size $O(r^3)$ and width $\Omega(1)$ such that any probability measure $\mu$ supported on $B_1 \cap (B_3-B_3)\cap B_4$ has $\|1_{A_j} \ast \mu\|_\infty \geq (1+\Omega(1))S_{i,j}$. Let $B^{(i+1)}$ have $\Lambda_{i+1}=\Lambda_i\cup\Lambda'$, and $\delta_{i+1}$ be the minimum of the widths of $B_3$, $B_1$, and $B_4$. This ensures that $B^{(i+1)} \subset B_1 \cap (B_3-B_3)\cap B_4$ and $d_{i+1}$ and $\delta_{i+1}$ all satisfy (\ref{cd0}) and (\ref{cd22}). The result is proved.
\end{proof}
There is a toy version of Theorem \ref{thm.maincore} in which $G$ is taken to be the additive group of the field $\F_{2^n}$. In this case one could choose $\phi(x)=x^2$ (where the multiplication is in $\F_{2^n}$) which is an automorphism of the additive group, and so we find that any $r$-colouring has $\exp(-r^{O(1)})4^n$ monochromatic triples $(x,y,z)$ with $x^2=y+z$. This was proved for finite fields of prime order in \cite{lin::2}, and is quite different to the analogous problem in the natural numbers resolved in \cite{grelin::}.

The advantage of toy versions of results like this is that they are simpler because they do not need Bohr sets (except in the sense that every subgroup is a Bohr set), and do not need any of localisation arguments. There is then a general method with its origins in \cite{bou::5} by which they can be converted to the non-toy setting, and this decoupling can help make the arguments easier to understand. For an introduction to these sorts of toys see the series \cite{gre::9,wol::3,Peluse:2024aa}.

Finally we turn to the proofs of Theorems \ref{thm.main2} \& \ref{thm.main3}.
\begin{theorem}[Theorem \ref{thm.main2}]
Suppose that $p$ is prime, $A$ is a $1 \times m$ matrix that satisfies the columns condition over $\Z/p\Z$, and $(\Z/p\Z)^*$ is $r$-coloured such that there is no monochromatic solution to $Ax=0$. Then $p=\exp((2r)^{O(1)})$.
\end{theorem}
\begin{proof}
Write $A=(\begin{array}{ccc}a_1 & \cdots  & a_m\end{array})$ with the entries (columns) ordered so that there is some $r \in \N^*$ such that $a_1+\cdots + a_r=0$ and $a_1 \neq 0$. Let $b:=-(a_{r+1}+\cdots + a_m)$. The hypothesis on the colouring means that there are no monochromatic solutions to $a_1x-a_1y = bz$, and hence no monochromatic solutions to $x-y=a_1^{-1}bz$. If $b=0$ then the lack of monochromatic solutions means that the colouring must be empty which is a contradiction. Hence $b \neq 0$. Let $G=\Z/p\Z$, $\phi(x)=a_1^{-1}bx$ which is an automorphism, and apply Theorem \ref{thm.maincore} to $G$ coloured with the $r$-colouring of $(\Z/p\Z)^*$ and an additional colour for $0$. Then by hypothesis the only monochromatic solution is $x=y=z=0$ and the Theorem tells us that $\exp(-(r+1)^{O(1)})|G|^2\leq 1$. Since $|G|=p$ this rearranges to give the conclusion.
\end{proof}

\begin{theorem}[Theorem \ref{thm.main3}]
Suppose that $a, N \in \N$ are coprime, and $(\Z/N\Z)^*$ is $r$-coloured such that there is no monochromatic solution to $ax_1+x_2-x_3=0$. Then $N=\exp((2r)^{O(1)})$.
\end{theorem}
\begin{proof} Since $a$ and $N$ are coprime the map $\phi:\Z/N\Z \rightarrow \Z/N\Z;x \mapsto ax$ is an automorphism. As above, colour $\Z/N\Z$ by taking the colouring of $(\Z/N\Z)^*$ and adding an extra colour class for $0$. We know that there is $1$ monochromatic solution to $\phi(x_1)=x_2-x_3$ or, equivalently, to $ax_1=x_2-x_3$, and by Theorem \ref{thm.maincore} this is at least $\exp(-(r+1)^{O(1)})N^2$. This rearranges to give the result.
\end{proof}

\bibliographystyle{halpha}

\bibliography{references}

\end{document}